\documentclass{amsproc}
\usepackage{breqn,float,amssymb,pdflscape,rotating}
\makeatletter
\@namedef{subjclassname@2020}{%
  \textup{2020} Mathematics Subject Classification}
\makeatother
\newtheorem{theorem}{Theorem}[section]

\theoremstyle{definition}

\newtheorem{example}[theorem]{Example}

\theoremstyle{remark}

\numberwithin{equation}{section}
\allowdisplaybreaks
\begin{document}

\title[Generalized double series]{A note on a generalized double series}


\author{Robert Reynolds}
\address[Robert Reynolds]{Department of Mathematics and Statistics, York University, Toronto, ON, Canada, M3J1P3}
\email[Corresponding author]{milver73@gmail.com}
\thanks{}




\date{}

\dedicatory{}

\begin{abstract}
By employing contour integration the derivation of a generalized double finite series involving the Hurwitz-Lerch zeta function is used to derive closed form formulae in terms of special functions. We use this procedure to find special cases of the summation and product formulae in terms of the Hurwitz-Lerch zeta function, trigonometric functions and the gamma function. A short table of quotient gamma functions is produced for easy reading by interested readers.
\end{abstract}

\maketitle
\section{Introduction}
%
Finite sums of special functions are recorded in chapter 5 in \cite{prud}. Equations (5.3.10.1) and (5.3.10.2) in \cite{prud3} yield finite summations of the Hurwitz-zeta function $\Phi(z,s,a)$ in terms of the Hurwitz-zeta function. In \cite{reyn_ejpam} an infinite sum involving the Hurwitz-Lerch zeta function was studied. The underlying formula transformed by contour integration is the infinite series involving the tangent function whose angle is a multiple of 2 raised to an integer power. The work done in \cite{reyn_ejpam} was centred around evaluating an infinite sum in terms of special functions, while in the current paper the author studies double finite series of special functions. In the current paper, the underlying function transformed by contour integration has a more generalized angle form which allows for more interesting evaluations in term of trigonometric and special functions.\\\\
The reason behind this research, is to expand on literature involving the finite double sum and product of special functions, such as the gamma function. The novelty this work represents, lies in the derivation and evaluation involving the double finite product of quotient gamma functions which is not present in current literature to the best of my knowledge. We use a contour integral method and apply it to a generalized double finite summation formula involving the secant function to study the resulting double finite summation of the Hurwitz-Lerch zeta function. Since the Hurwitz-Lerch zeta function is a special function it has several composite functions which can yield new results by algebraic means. Some of these composite functions are the gamma function and other trigonometric functions. The objectives pursued in this work were the derivation of a double finite sum involving the Hurwitz-Lerch zeta function, along with double finite products involving the gamma function and a few functional identities for the Hurwitz-Lerch zeta function.\\\\
In this work we apply the contour integral method \cite{reyn4}, to the the generalized form of the finite secant sum on page 90 in \cite{hobson} to derive a generalized form involving the finite sum of the secant function, resulting in
\begin{multline}\label{eq:contour}
\frac{1}{2\pi i}\int_{C}\sum_{p=0}^{n}a^w w^{-1-k} \left(\left(1+(-1)^{1+z}\right) \sec \left((m+w) (1+2
   z)^p\right)\right. \\ \left.
   +\sum_{j=1}^{z}2 (-1)^{1+j+z} \cos \left(2 j (m+w) (1+2 z)^{-1+p}\right)\right. \\ \left. \sec
   \left((m+w) (1+2 z)^p\right)\right)dw\\
=\frac{1}{2\pi i}\int_{C}a^w w^{-k-1} \left(\sec
   \left((m+w) (2 z+1)^n\right)-\sec \left(\frac{m+w}{2 z+1}\right)\right)dw
\end{multline}
where $a,m,k\in\mathbb{C},Re(m+w)>0,n\in\mathbb{Z^{+}},z\in\mathbb{Z^{+}}$. Using equation (\ref{eq:contour}) the main Theorem to be derived and evaluated is given by
\begin{multline}
\sum_{p=0}^{n}\left((-1)^{z+1}+1\right) \left((2 z+1)^p\right)^k e^{i m (2 z+1)^p} \Phi \left(-e^{2 i m (2
   z+1)^p},-k,\frac{1}{2} \left(a (2 z+1)^{-p}+1\right)\right)\\
+\sum_{p=0}^{n}\sum_{j=1}^{z}(-1)^{j+z+1} \left((2 z+1)^p\right)^k e^{i m (-2 j+2
   z+1) (2 z+1)^{p-1}}\\
 \left(\Phi \left(-e^{2 i m (2 z+1)^p},-k,\frac{1}{2} \left(a (2 z+1)^{-p}+1-\frac{2 j}{2
   z+1}\right)\right)\right. \\ \left.
+e^{4 i j m (2 z+1)^{p-1}} \Phi \left(-e^{2 i m (2 z+1)^p},-k,\frac{j}{2 z+1}+\frac{1}{2} \left(a(2 z+1)^{-p}+1\right)\right)\right)\\
=\left((2 z+1)^n\right)^k e^{i m (2 z+1)^n} \Phi \left(-e^{2 i m (2z+1)^n},-k,\frac{1}{2} \left(a (2 z+1)^{-n}+1\right)\right)\\
-\left(\frac{1}{2 z+1}\right)^k e^{\frac{i m}{2 z+1}}
   \Phi \left(-e^{\frac{2 i m}{2 z+1}},-k,\frac{1}{2} (2 z a+a+1)\right)
\end{multline}
where the variables $k,a,m$ are general complex numbers and $n,z$ are positive integers. The derivations follow the method used by us in \cite{reyn4}. This method involves using a form of the generalized Cauchy's integral formula given by
\begin{equation}\label{intro:cauchy}
\frac{y^k}{\Gamma(k+1)}=\frac{1}{2\pi i}\int_{C}\frac{e^{wy}}{w^{k+1}}dw,
\end{equation}
where $y,w\in\mathbb{C}$ and $C$ is in general an open contour in the complex plane where the bilinear concomitant \cite{reyn4} is equal to zero at the end points of the contour. This method involves using a form of equation (\ref{intro:cauchy}) then multiplies both sides by a function, then takes the double finite sum of both sides. This yields a double finite sum in terms of a contour integral. Then we multiply both sides of equation (\ref{intro:cauchy}) by another function and take the infinite sum of both sides such that the contour integral of both equations are the same.
\section{The Hurwitz-Lerch zeta function}
We use equation (1.11.3) in \cite{erd} where $\Phi(z,s,v)$ is the Lerch function which is a generalization of the Hurwitz zeta $\zeta(s,v)$ and Polylogarithm functions $Li_{n}(z)$. In number theory and complex analysis, the Lerch function is a mathematical function that appears in many branches of mathematics and physics. It is named after Czech mathematician Mathias Lerch, who published a paper about the function in 1887. Numerous areas of mathematics, including number theory (especially in the investigation of the Riemann zeta function and its generalizations), complex analysis, and theoretical physics, all have uses for it. It can be used to express a variety of complex functions and series and is involved in numerous mathematical identities. The Lerch function has a series representation given by
\begin{equation}\label{knuth:lerch}
\Phi(z,s,v)=\sum_{n=0}^{\infty}(v+n)^{-s}z^{n}
\end{equation}
where $|z|<1, v \neq 0,-1,-2,-3,..,$ and is continued analytically by its integral representation given by
\begin{equation}\label{knuth:lerch1}
\Phi(z,s,v)=\frac{1}{\Gamma(s)}\int_{0}^{\infty}\frac{t^{s-1}e^{-(v-1)t}}{e^{t}-z}dt
\end{equation}
where $Re(v)>0$, and either $|z| \leq 1, z \neq 1, Re(s)>0$, or $z=1, Re(s)>1$.
\section{Contour integral representation for the finite sum of the Hurwitz-Lerch zeta functions}
In this section we derive the contour integral representations of the left-hand side and right-hand side of equation (\ref{eq:contour}) in terms of the Hurwtiz-Lerch zeta and trigonometric functions.
\subsection{Derivation of the generalized Hurwitz-Lerch contour integral for the secant function}
We use the method in \cite{reyn4}. Using equation (\ref{intro:cauchy})  we first replace $y$ by $\log (a)+i b (2 y+1)$ and multiply both sides by $-2 i (-1)^y e^{i b m (2 y+1)}$ then take the infinite sum over $y\in [0,\infty)$ and simplify in terms of the Hurwitz-Lerch zeta function to get
\begin{multline}\label{gen_sec}
-\frac{i 2^{k+1} (i b)^k e^{i b m} \Phi \left(-e^{2 i b m},-k,\frac{b-i \log (a)}{2 b}\right)}{\Gamma(k+1)}\\
   =\frac{1}{2\pi i}\sum_{y=0}^{\infty}\int_{C}(-1)^y a^w w^{-k-1} e^{i b (2 y+1) (m+w)}dw\\
   =\frac{1}{2\pi i}\int_{C}\sum_{y=0}^{\infty}(-1)^y a^w w^{-k-1} e^{i b (2 y+1) (m+w)}dw\\
   =-\frac{1}{2\pi i}\int_{C}i a^w w^{-k-1} \sec (b (m+w))dw
\end{multline}
from equation (1.232.2) in \cite{grad} where $Re(w+m)>0$ and $Im\left(m+w\right)>0$ in order for the sums to converge. We apply Tonelli's theorem for multiple sums, see page 177 in \cite{gelca} as the summands are of bounded measure over the space $\mathbb{C} \times [0,\infty)$.
\subsection{Derivation of a generalized Hurwitz-Lerch zeta function in terms of the product of the cosine and secant function contour integral}
We use the method in \cite{reyn4}. Using a generalization of Cauchy's integral formula (\ref{intro:cauchy}) we first replace $y$ by $ \log (a)+i x+y$ then multiply both sides by $e^{mxi}$ then form a second equation by replacing $x$ by $-x$ and add both equations to get
\begin{multline}
\frac{e^{-i m x} \left(e^{2 i m x} (\log (a)+i x+y)^k+(\log (a)-i
   x+y)^k\right)}{\Gamma(k+1)}\\
   =\frac{1}{2\pi i}\int_{C}2 w^{-k-1} e^{w (\log (a)+y)} \cos (x (m+w))dw
\end{multline}
Next we replace $y$ by $i b (2 y+1)$ and multiply both sides by $(-1)^y e^{i b m (2 y+1)}$ and take the infinite sum over $y\in[0,\infty)$ and simplify in terms of the Hurwitz-Lerch zeta function to get
\begin{multline}\label{gen_cos_sec}
\frac{2^k (i b)^k e^{i m (b-x)} \left(\Phi \left(-e^{2 i b m},-k,\frac{b-x-i \log (a)}{2 b}\right)+e^{2 i m x}
   \Phi \left(-e^{2 i b m},-k,\frac{b+x-i \log (a)}{2 b}\right)\right)}{\Gamma(k+1)}\\
   =\frac{1}{2\pi i}\sum_{y=0}^{\infty}\int_{C}2 (-1)^y a^w w^{-k-1} e^{i b (2 y+1) (m+w)} \cos (x (m+w))dw\\
   =\frac{1}{2\pi i}\int_{C}\sum_{y=0}^{\infty}2 (-1)^y a^w w^{-k-1} e^{i b (2 y+1) (m+w)} \cos (x (m+w))dw\\
   =\frac{1}{2\pi i}\int_{C}a^w w^{-k-1} \sec (b (m+w)) \cos (x (m+w))dw
   \end{multline}
from equation (1.232.2) and (1.411.3) in \cite{grad} where $Re(w+m)>0$ and $Im\left(m+w\right)>0$ in order for the sums to converge. We apply Tonelli's theorem for sums and integrals, see page 177 in \cite{gelca} as the summand and integral are of bounded measure over the space $\mathbb{C} \times [0,\infty)$.
\subsection{Derivation of the contour integrals}
In this section we will use equations (\ref{gen_sec}) and (\ref{gen_cos_sec}) and simple substitutions to derive the contour integrals in equation (\ref{eq:contour}).
\subsubsection{Left-hand side first contour integral}
Use equation (\ref{gen_sec}) and replace $b$ by $(2 z+1)^p$ then multiply both sides by $(-1)^{z+1}+1$ and take the finite sum over $p\in[0,n]$ to get;
\begin{multline}\label{lsfc}
\sum_{p=0}^{n}\frac{1}{k!}2^{k+1} \left((-1)^{z+1}+1\right) \left(i (2 z+1)^p\right)^k e^{i m(2 z+1)^p}\\
 \Phi \left(-e^{2 i m (2 z+1)^p},-k,\frac{1}{2} (2 z+1)^{-p}
   \left((2 z+1)^p-i \log (a)\right)\right)\\
   =\frac{1}{2\pi i}\int_{C}\sum_{p=0}^{n}\left((-1)^{z+1}+1\right) a^w
   w^{-k-1} \sec \left((m+w) (2 z+1)^p\right)dw
\end{multline}
\subsubsection{Left-hand side second contour integral}
Use equation (\ref{gen_cos_sec}) and replace $x$ by $2 j (2 z+1)^{p-1}$, and $b$ by $(2 z+1)^p$ then multiply both sides by $2 (-1)^{j+z+1}$ and take the finite sums over $p\in[0,n]$ and $j\in[1,z]$ to get;
\begin{multline}\label{lssc}
\sum_{p=0}^{n}\frac{1}{k!}2^{k+1} (-1)^{j+z+1} \left(i (2 z+1)^p\right)^k \exp \left(i m
   \left((2 z+1)^p-2 j (2 z+1)^{p-1}\right)\right)\\
    \sum_{j=1}^{z}\left(\Phi \left(-e^{2 i m (2
   z+1)^p},-k,\frac{1}{2} (2 z+1)^{-p} \left(-2 j (2 z+1)^{p-1}+(2 z+1)^p-i \log
   (a)\right)\right)\right. \\ \left.
   +e^{4 i j m (2 z+1)^{p-1}} \Phi \left(-e^{2 i m (2
   z+1)^p},-k,\frac{1}{2} (2 z+1)^{-p} \left(2 j (2 z+1)^{p-1}+(2 z+1)^p-i \log
   (a)\right)\right)\right)\\
   =\frac{1}{2\pi i}\int_{C}\sum_{p=0}^{n}\sum_{j=1}^{z}2 a^w (-1)^{j+z+1} w^{-k-1} \sec \left((m+w) (2
   z+1)^p\right) \cos \left(2 j (m+w) (2 z+1)^{p-1}\right)dw
\end{multline}
\subsubsection{Right-hand side first contour integral}
Use equation (\ref{gen_sec}) and replace $b$ by $\frac{1}{2 z+1}$ then multiply both sides by $-1$ to get;
\begin{multline}\label{rsfc}
-\frac{1}{k!}2^{k+1} \left(\frac{i}{2 z+1}\right)^k e^{\frac{i m}{2 z+1}} \Phi
   \left(-e^{\frac{2 i m}{2 z+1}},-k,\frac{1}{2} (2 z+1) \left(\frac{1}{2 z+1}-i
   \log (a)\right)\right)\\
   =-\frac{1}{2\pi i}\int_{C}a^w w^{-k-1} \sec \left(\frac{m+w}{2z+1}\right)dw
\end{multline}

\subsubsection{Right-hand side second contour integral}
Use equation (\ref{gen_sec}) and replace $b$ by $(2 z+1)^n$ to get;
\begin{multline}\label{rssc}
\frac{1}{k!}2^{k+1} \left(i (2 z+1)^n\right)^k e^{i m (2 z+1)^n} \Phi
   \left(-e^{2 i m (2 z+1)^n},-k,\frac{1}{2} (2 z+1)^{-n} \left((2 z+1)^n-i \log
   (a)\right)\right)\\
   =\frac{1}{2\pi i}\int_{C}a^w w^{-k-1} \sec \left((m+w) (2 z+1)^n\right)dw
\end{multline}
%
%
\section{Finite sum of the Hurwitz-Lerch zeta function and evaluations}
In this section we will derive and evaluate formulae involving the finite double sum of the Hurwitz-Lerch zeta function in terms other special functions, trigonometric functions and fundamental constants.
\begin{theorem}
For all $k,a\in\mathbb{C},-1< Re(m)<1,n\in\mathbb{Z_{+}},z\in\mathbb{Z_{+}}$ then,
\begin{multline}\label{eq:theorem}
\sum_{p=0}^{n}\left((-1)^{z+1}+1\right) \left((2 z+1)^p\right)^k e^{i m (2 z+1)^p} \Phi \left(-e^{2 i m (2
   z+1)^p},-k,\frac{1}{2} \left(a (2 z+1)^{-p}+1\right)\right)\\
+\sum_{p=0}^{n}\sum_{j=1}^{z}(-1)^{j+z+1} \left((2 z+1)^p\right)^k e^{i m (-2 j+2
   z+1) (2 z+1)^{p-1}}\\
 \left(\Phi \left(-e^{2 i m (2 z+1)^p},-k,\frac{1}{2} \left(a (2 z+1)^{-p}+1-\frac{2 j}{2
   z+1}\right)\right)\right. \\ \left.
+e^{4 i j m (2 z+1)^{p-1}} \Phi \left(-e^{2 i m (2 z+1)^p},-k,\frac{j}{2 z+1}+\frac{1}{2} \left(a(2 z+1)^{-p}+1\right)\right)\right)\\
=\left((2 z+1)^n\right)^k e^{i m (2 z+1)^n} \Phi \left(-e^{2 i m (2z+1)^n},-k,\frac{1}{2} \left(a (2 z+1)^{-n}+1\right)\right)\\
-\left(\frac{1}{2 z+1}\right)^k e^{\frac{i m}{2 z+1}}
   \Phi \left(-e^{\frac{2 i m}{2 z+1}},-k,\frac{1}{2} (2 z a+a+1)\right)
\end{multline}
\end{theorem}
\begin{proof}
Observe that the addition of the right-hand sides of equations (\ref{lsfc}) and (\ref{lssc}), is equal to the addition of the right-hand sides of equations (\ref{rsfc}) and (\ref{rssc}) so we may equate the left-hand sides and simplify relative to equation (\ref{eq:contour}) and the Gamma function to yield the stated result.
\end{proof}
\begin{example}
The degenerate case.
\begin{multline}
\sum_{p=0}^{n}\sec \left(m (2 z+1)^p\right) \left(\sum_{j=1}^{z}2 e^{i \pi  (j+z)} \cos \left(2 j m (2 z+1)^{p-1}\right)+e^{i \pi 
   z}-1\right)\\
   =\sec \left(\frac{m}{2 z+1}\right)-\sec \left(m (2 z+1)^n\right)
\end{multline}
\end{example}
\begin{proof}
Use equation (\ref{eq:theorem}) and set $k=0$ and simplify using entry (2) in Table below equation (64:12:7) in \cite{atlas}. 
\end{proof}
\begin{example}
Double finite product of quotient gamma functions, where $a\in\mathbb{C},n\in\mathbb{Z_{+}}$, $z$ is even. Multiple finite products involving the gamma function are recorded in [p. 214] in \cite{krysicki} and  [p.173] in \cite{kal}.
\begin{multline}\label{eq1}
\prod_{p=0}^{n}\prod_{j=1}^{z}\left(\frac{\Gamma \left(\frac{1}{4} \left(a (2 z+1)^{-p}+1-\frac{2 j}{2 z+1}\right)\right) \Gamma
   \left(\frac{1}{4} \left(a (2 z+1)^{-p}+1+\frac{2 j}{2 z+1}\right)\right)}{\Gamma \left(\frac{1}{4} \left(a (2
   z+1)^{-p}+3-\frac{2 j}{2 z+1}\right)\right) \Gamma \left(\frac{1}{4} \left(a (2 z+1)^{-p}+3+\frac{2 j}{2
   z+1}\right)\right)}\right)^{(-1)^j}\\
   =\frac{(2 z+1)^{\frac{n+1}{2}} \Gamma \left(\frac{1}{4} (2 z a+a+1)\right)
   \Gamma \left(\frac{1}{4} \left(a (2 z+1)^{-n}+3\right)\right)}{\Gamma \left(\frac{1}{4} (2 z a+a+3)\right) \Gamma
   \left(\frac{1}{4} \left(a (2 z+1)^{-n}+1\right)\right)}
\end{multline}
\end{example}
\begin{proof}
Use equation (\ref{eq:theorem}) and set $m=0$ and simplify in terms of the Hurwitz zeta function using entry (4) in Table below (64:12:7) in \cite{atlas}. Next take the first partial derivative with respect to $k$ and set $k=0$ and simplify in terms of the log-gamma function using equation (25.11.18) in \cite{dlmf}. Next we take the exponential of both sides and simplify in terms of the gamma function.
\end{proof}
\begin{example}
Product of gamma functions in terms of the square root of odd numbers greater than equal to 5. We also present a product formula for the ratio the Elliptic integral of the first kind see [p.1140] in \cite{eric}.
\begin{multline}\label{eq2}
\frac{\Gamma \left(\frac{a+1}{4}\right) \Gamma \left(\frac{1}{4} (2 z a+a+3)\right) }{\Gamma \left(\frac{a+3}{4}\right) \Gamma
   \left(\frac{1}{4} (2 z a+a+1)\right)}\prod_{j=1}^{z}\left(\frac{\Gamma
   \left(\frac{1}{4} \left(a-\frac{2 j}{2 z+1}+1\right)\right) \Gamma \left(\frac{1}{4} \left(a+\frac{2 j}{2
   z+1}+1\right)\right)}{\Gamma \left(\frac{1}{4} \left(a-\frac{2 j}{2 z+1}+3\right)\right) \Gamma \left(\frac{1}{4}
   \left(a+\frac{2 j}{2 z+1}+3\right)\right)}\right)^{(-1)^j}\\
   =\sqrt{2 z+1}
\end{multline}
\end{example}
\begin{proof}
Use equation (\ref{eq1}) and set $p=n=0$ and simplify.
\end{proof}
\begin{example}
A Hurwitz-Lerch zeta functional equation with first parameter involving the cubic root of z. 
\begin{multline}\label{eq3}
\Phi \left(-(i z)^{2/3},s,a\right)\\
=9^{-s} \left(3^s \Phi \left(z^2,s,\frac{a}{3}\right)+(i z)^{2/3} \left(\Phi
   \left(z^6,s,\frac{a+1}{9}\right)-2\times 3^s \Phi \left(z^2,s,\frac{a+1}{3}\right)\right.\right. \\ \left.\left.
   +3^s (i z)^{2/3} \Phi
   \left(z^2,s,\frac{a+2}{3}\right)+z^4 \Phi \left(z^6,s,\frac{a+7}{9}\right)+z^2 \Phi
   \left(z^6,s,\frac{a+4}{9}\right)\right)\right)
\end{multline}
\end{example}
\begin{proof}
Use equation (\ref{eq:theorem}) and set $n= 1,z= 1,m= -i \log (z),a= e^{i a}$ and simplify. Next replace $a= \frac{2}{3} \left(a-\frac{1}{2}\right),k= -s,z=zi$ and simplify.
\end{proof}
\begin{example}
Double finite sum in terms of trigonometric functions. 
\begin{multline}\label{eq4}
\sum_{p=0}^{n}\sum_{j=1}^{z}2 e^{i \pi  (j+z)} (2 z+1)^{p-1} \sec \left(x (2 z+1)^p\right)\\
 \left((2 z+1) \tan \left(x (2 z+1)^p\right) \cos\left(2 j x (2 z+1)^{p-1}\right)-2 j \sin \left(2 j x (2 z+1)^{p-1}\right)\right)\\
+\sum_{p=0}^{n}\left(1-e^{i \pi  z}\right) (2
   z+1)^p \tan \left(x (2 z+1)^p\right) \sec \left(x (2 z+1)^p\right)\\
=\frac{1}{8 z+4}\left(\sec ^2\left(\frac{x}{2 z+1}\right)
   \left((2 z+1)^{n+1} \left(2 \sin \left(x (2 z+1)^n\right)-\sin \left(x \left(\frac{2}{2 z+1}-(2
   z+1)^n\right)\right)\right.\right.\right. \\ \left.\left.\left.
+\sin \left(x \left((2 z+1)^n+\frac{2}{2 z+1}\right)\right)\right)-\sin \left(x
   \left(\frac{1}{2 z+1}-2 (2 z+1)^n\right)\right)\right.\right. \\ \left.\left.
-\sin \left(x \left(2 (2 z+1)^n+\frac{1}{2 z+1}\right)\right)-2 \sin
   \left(\frac{x}{2 z+1}\right)\right) \sec ^2\left(x (2 z+1)^n\right)\right)
\end{multline}
\end{example}
\begin{proof}
Use equation (\ref{eq:theorem}) and set $k=1,a=1,m=x$ and apply the method in section (8) in \cite{reyn_ejpam}.
\end{proof}
\begin{example} 
A functional equation involving the Hurwitz-Lerch zeta function.
\begin{multline}\label{eq5}
\Phi \left(-\sqrt[3]{-z},s,a\right)\\
=3^{-s} \left(\Phi \left(z,s,\frac{a}{3}\right)-\sqrt[3]{-z} \Phi
   \left(z,s,\frac{a+1}{3}\right)+(-z)^{2/3} \Phi \left(z,s,\frac{a+2}{3}\right)\right)
\end{multline}
\end{example}
\begin{proof}
Use equation (\ref{eq:theorem}) and set $n= 0,z= 1,m= \log (z)/(2i),a= e^{i a}$ and simplify. Next replace $a= \frac{2}{3} \left(a-\frac{1}{2}\right),k= -s$ and simplify.
\end{proof}
\begin{example} 
Finite product involving the product of trigonometric functions.
\begin{multline}
\prod_{p=0}^{n}\prod_{j=1}^{z}\exp \left(i (-1)^{j+z} (2 z+1)^{-p} e^{i m (-2 j+2 z+1) (2 z+1)^{p-1}}
   \left(\Phi \left(-e^{2 i m (2 z+1)^p},1,\frac{1}{2}-\frac{j}{2 z+1}\right)\right.\right. \\ \left.\left.
+e^{4 i jm (2 z+1)^{p-1}} \Phi \left(-e^{2 i m (2 z+1)^p},1,\frac{j}{2z+1}+\frac{1}{2}\right)\right)\right)\\
\prod_{p=0}^{n}\left(1-i e^{i m (2
   z+1)^p}\right)^{\left((-1)^{z+1}+1\right) (2 z+1)^{-p}} \left(1+i e^{i m (2
   z+1)^p}\right)^{\left((-1)^z-1\right) (2 z+1)^{-p}}\\
=\left(i \cot \left(\frac{2 m+2
   \pi  z+\pi }{8 z+4}\right)\right)^{2 z+1} \left(-i \tan \left(\frac{1}{4} \left(2 m
   (2 z+1)^n+\pi \right)\right)\right)^{(2 z+1)^{-n}}
\end{multline}
\end{example}
\begin{proof}
Use equation (\ref{eq:theorem}) and set $k=-1,a=1$ and simplify using entry (3) in Table below (64:12:7) in atlas.
\end{proof}
\section{Special cases involving the double finite sum of the Hurwitz-Lerch zeta function}
In this section we derive special cases of the Hurwitz-Lerch zeta function  in terms of special functions and Catalan's constant $C$.
\begin{example} 
Double finite sum in terms of Catalan's constant $C$. 
\begin{multline}\label{eq_catalan}
\sum_{p=0}^{n}\sum_{j=1}^{z}(2 z+1)^{-2 p} \left(16 C \left((-1)^z-1\right)+(-1)^{j+z} \left(\psi
   ^{(1)}\left(\frac{-2 j+2 z+1}{8 z+4}\right)\right.\right. \\ \left.\left. 
   +\psi ^{(1)}\left(\frac{2 j+2 z+1}{8
   z+4}\right)-\psi ^{(1)}\left(\frac{-2 j+6 z+3}{8 z+4}\right)-\psi ^{(1)}\left(\frac{2 j+6
   z+3}{8 z+4}\right)\right)\right)\\
   =16 C \left((2 z+1)^2-(2 z+1)^{-2 n}\right)
\end{multline}
\end{example}
\begin{proof}
Use equation (\ref{eq:theorem}) and set $k=-2,a=1,m=0$ and simplify using equation (4) in \cite{guillera}.
\end{proof}
\begin{example} 
Finite sum in terms of Catalan's constant $C$ and the digamma function $\psi ^{(1)}(s)$. 
\begin{multline}
\sum_{j=1}^{z}\frac{(2 z+1)^2 (-1)^{j+z} }{4 z(z+1)}\left(\psi ^{(1)}\left(\frac{-2 j+2 z+1}{8 z+4}\right)+\psi ^{(1)}\left(\frac{2 j+2 z+1}{8z+4}\right)\right. \\ \left.
-\psi ^{(1)}\left(\frac{-2 j+6 z+3}{8 z+4}\right)-\psi ^{(1)}\left(\frac{2 j+6 z+3}{8 z+4}\right)\right)\\
   =\frac{4 C (2 z+1)^2 \left(4 z-(-1)^z+5\right)}{z+1}
\end{multline}
\end{example}
\begin{proof}
Use equation (\ref{eq_catalan}) and take the limit as $n\to \infty$ and simplify.
\end{proof}
\begin{example} 
Finite sum in terms of the first partial derivative of the Hurwitz-Lerch zeta function.
\begin{multline}
\sum_{j=1}^{z}(-1)^j \left(\Phi'\left(-1,-1,\frac{1}{2}-\frac{j}{2
   z+1}\right)+\Phi'\left(-1,-1,\frac{j}{2 z+1}+\frac{1}{2}\right)\right)\\
=\frac{C \left(-2
   z+(-1)^{-z}-1\right)}{\pi  (2 z+1)}
\end{multline}
\end{example}
\begin{proof}
Use equation (\ref{eq:theorem}) and take the first partial derivative with respect to $k$ and set $k=1,m=0,a=1$ and simplify using equation (20) in \cite{guillera}.
\end{proof}
\begin{example} 
Double finite sum involving Euler's constant $\gamma$. 
\begin{multline}\label{eq:euler}
\sum_{p=0}^{n}\sum_{j=1}^{z}2 i e^{i \pi  (j+z)} (2 z+1)^{-p} \left(-\Phi'\left(-1,1,\frac{1}{2}-\frac{j}{2
   z+1}\right)-\Phi'\left(-1,1,\frac{j}{2 z+1}+\frac{1}{2}\right)\right. \\ \left.
   +\pi  \sec \left(\frac{\pi  j}{2z+1}\right) \log \left(i (2 z+1)^p\right)\right)-i \pi  \left((-1)^z-1\right) \sum_{p=0}^{n}(2 z+1)^{-p} \log \left(-\frac{i e^{\gamma } \pi
   ^3 (2 z+1)^{-p}}{32 \Gamma \left(\frac{5}{4}\right)^4}\right)\\
   =-i \pi  (2 z+1)^{-n} \left((2 z+1)^{n+1} \log \left(-\frac{8 i
   e^{\gamma } \pi ^3 (2 z+1)}{\Gamma \left(\frac{1}{4}\right)^4}\right)+\log \left(\frac{32 i e^{-\gamma } \Gamma
   \left(\frac{5}{4}\right)^4 (2 z+1)^n}{\pi ^3}\right)\right)
\end{multline}
\end{example}
\begin{proof}
Use equation (\ref{eq:theorem}) and take the first partial derivative with respect to $k$ and set $k=-1,m=0,a=1$ and simplify using equation (53) in \cite{guillera}.
\end{proof}
\begin{example} 
The exponential of the Hurwitz-Lerch zeta function in terms of the log-gamma function.
\begin{multline}
\exp \left(-\frac{13\left(\Phi'\left(-1,1,\frac{1}{6}\right)+\Phi'\left(-1,1,\frac{5}{6}\right)\right)}{54 \pi }\right)=\frac{2^{11/27} e^{-13 \gamma /27} \Gamma \left(\frac{1}{4}\right) \Gamma
   \left(\frac{5}{4}\right)^{25/27}}{3^{13/36} \pi ^{13/9}}
\end{multline}
\end{example}
\begin{proof}
Use equation (\ref{eq:euler}) and set $n=2,z=1$ and simplify.
\end{proof}
\begin{example} 
A double finite sum involving Catalan's constant $C$.
\begin{multline}
\sum_{p=0}^{n}4 \left((-1)^z-1\right) (2 z+1)^{-2 p} \left(-\Phi'\left(-1,2,\frac{1}{2}\right)+4 C \log
   \left(i (2 z+1)^p\right)\right)\\
   +\sum_{p=0}^{n}\sum_{j=1}^{z}(-1)^{j+z} (2 z+1)^{-2 p} \left(-4
   \Phi'\left(-1,2,\frac{1}{2}-\frac{j}{2 z+1}\right)-4
   \Phi'\left(-1,2,\frac{j}{2 z+1}+\frac{1}{2}\right)\right. \\ \left.
   +\left(\psi ^{(1)}\left(\frac{-2 j+2 z+1}{8
   z+4}\right)+\psi ^{(1)}\left(\frac{2 j+2 z+1}{8 z+4}\right)-\psi ^{(1)}\left(\frac{-2 j+6 z+3}{8 z+4}\right)\right.\right. \\ \left.\left.
   -\psi
   ^{(1)}\left(\frac{2 j+6 z+3}{8 z+4}\right)\right) \log \left(i (2 z+1)^p\right)\right)\\
   =4 (2 z+1)^{-2 n}
   \left(\Phi'\left(-1,2,\frac{1}{2}\right)-4 C \log \left(i (2 z+1)^n\right)\right)\\
   +4 (2 z+1)^2
   \left(-\Phi'\left(-1,2,\frac{1}{2}\right)+4 C \log \left(\frac{i}{2 z+1}\right)\right)
\end{multline}
\end{example}
\begin{proof}
Use equation (\ref{eq:theorem}) and take the first partial derivative with respect to $k$ and set $k=-2,m=0,a=1$ and simplify using equation (4) in \cite{guillera}.
\end{proof}
\begin{example} 
Double finite product involving tangent and cotangent functions.
\begin{multline}
\prod_{p=0}^{n}\prod_{j=1}^{z}\exp \left(\frac{\pi  i (-1)^{j+z} \sec \left(\frac{j \pi }{1+2 z}\right)}{(1+2 z)^p}\right) \prod_{p=0}^{n}\tan^{\frac{-1+(-1)^z}{(1+2 z)^p}}\left(\frac{1}{4} \pi  \left(1+2 (1+2 z)^{1+p}\right)\right)\\
=i
   \left(-(-1)^{3/4}\right)^{\frac{(1+2 z)^{-n} \left(-1+(-1)^z-\left(-1+(-1)^z\right) (1+2 z)^{1+n}\right)}{z}} (-1)^z
   \left(i \cot \left(\frac{1}{4} \pi  \left(1+2 (1+2 z)^{1+n}\right)\right)\right)^{(1+2 z)^{-n}}
\end{multline}
\end{example}
\begin{proof}
Use equation (\ref{eq:theorem}) and set $k= -1,m= \pi  (2 z+1),a= 1$ and simplify using entry (1) in Table below (64:12:7) in \cite{atlas}.
\end{proof}
\begin{example} 
Double finite product involving the tangent function.
\begin{multline}
\prod_{p=0}^{n}\prod_{j=1}^{z}\exp \left(-\frac{i (-1)^{j+z}}{(1+2 z)^p} \left(e^{i \pi  (1+2 z)^{-1-n+p} (1-2 j+2 z)} \Phi \left(-e^{2 i \pi  (1+2z)^{-n+p}},1,\frac{1}{2}-\frac{j}{1+2 z}\right)\right.\right. \\ \left.\left.
+e^{i \pi  (1+2 z)^{-1-n+p} (1+2 j+2 z)} \Phi \left(-e^{2 i \pi  (1+2z)^{-n+p}},1,\frac{1}{2}+\frac{j}{1+2 z}\right)\right)\right)\\
\prod_{p=0}^{n} \left(-i \tan \left(\frac{1}{4} \pi \left(1+2 (1+2 z)^{-n+p}\right)\right)\right)^{\frac{-1+(-1)^z}{(1+2 z)^p}}\\
=-e^{-\frac{\pi  i}{2 (1+2 z)^n}}
   \left(-i \tan \left(\frac{1}{4} \pi  \left(1+2 (1+2 z)^{-1-n}\right)\right)\right)^{2 z} \left(i \tan
   \left(\frac{1}{4} \pi  \left(1+2 (1+2 z)^{-1-n}\right)\right)\right)
\end{multline}
\end{example}
\begin{proof}
Use equation (\ref{eq:theorem}) and set $k= -1,m= \pi  (2 z+1)^{-n},a= 1$ and simplify using entry (1) in Table below (64:12:7) in \cite{atlas}.
\end{proof}
\begin{example} 
A double finite product involving the gamma function and table of special cases. 
\begin{multline}\label{eq:gamma}
\prod_{p=0}^{n}\prod_{j=1}^{z}\left(i (2 z+1)^p\right)^{(-1)^{j+z}+\frac{1}{2} \left((-1)^z-1\right)} \left(\frac{\Gamma \left(\frac{-2 j+6
   z+3}{8 z+4}\right) \Gamma \left(\frac{j}{4 z+2}+\frac{3}{4}\right)}{\Gamma \left(\frac{-2 j+2 z+1}{8 z+4}\right)
   \Gamma \left(\frac{j}{4 z+2}+\frac{1}{4}\right)}\right)^{(-1)^{j+z}}\\
=(2 z+1)^{\frac{1}{2} (-n-1)} \left(\frac{3
   \Gamma \left(-\frac{3}{4}\right)}{\Gamma \left(-\frac{1}{4}\right)}\right)^{(-1)^z-1}
\end{multline}
\end{example}
\begin{proof}
Use equation (\ref{eq:theorem}) and set $m=0,a=1$ and simplify in terms of the Hurwitz zeta function using entry (4) in Table below (64:12:7) in \cite{atlas}. Next take the first partial derivative with respect to $k$ and set $k=0$ and simplify in terms of the log-gamma function using equation (25.11.18) in \cite{dlmf}.
\end{proof}
\newpage
\section{Table of quotient gamma functions}
In this section we look at simplified forms of equation (\ref{eq:gamma}).
\begin{center}
\begin{table}[h]
\setlength{\tabcolsep}{9pt} 
\renewcommand{\arraystretch}{2.7} 
\begin{tabular}{ l | c r }
  \hline
  $\frac{\Gamma \left(\frac{3}{20}\right) \Gamma \left(\frac{7}{20}\right) \Gamma \left(\frac{11}{20}\right)
   \Gamma \left(\frac{19}{20}\right)}{\Gamma \left(\frac{1}{20}\right) \Gamma \left(\frac{9}{20}\right) \Gamma
   \left(\frac{13}{20}\right) \Gamma \left(\frac{17}{20}\right)}$ & $\frac{1}{\sqrt{5}}$  \\
  $\frac{\Gamma \left(\frac{3}{20}\right)^2 \Gamma \left(\frac{7}{20}\right)^2 \Gamma \left(\frac{11}{20}\right)^2
   \Gamma \left(\frac{19}{20}\right)^2}{\Gamma \left(\frac{1}{20}\right)^2 \Gamma \left(\frac{9}{20}\right)^2 \Gamma
   \left(\frac{13}{20}\right)^2 \Gamma \left(\frac{17}{20}\right)^2}$ & $\frac{1}{5}$ \\
  $\frac{\Gamma \left(\frac{1}{12}\right)^2 \Gamma \left(\frac{7}{36}\right)^2 \Gamma \left(\frac{11}{36}\right)^2
   \Gamma \left(\frac{5}{12}\right)^2 \Gamma \left(\frac{19}{36}\right)^2 \Gamma \left(\frac{23}{36}\right)^2 \Gamma
   \left(\frac{31}{36}\right)^2 \Gamma \left(\frac{35}{36}\right)^2}{\Gamma \left(\frac{1}{36}\right)^2 \Gamma
   \left(\frac{5}{36}\right)^2 \Gamma \left(\frac{13}{36}\right)^2 \Gamma \left(\frac{17}{36}\right)^2 \Gamma
   \left(\frac{7}{12}\right)^2 \Gamma \left(\frac{25}{36}\right)^2 \Gamma \left(\frac{29}{36}\right)^2 \Gamma
   \left(\frac{11}{12}\right)^2}$ & $\frac{1}{9}$ \\
  $\frac{\Gamma \left(\frac{1}{12}\right)^3 \Gamma \left(\frac{7}{36}\right)^3 \Gamma \left(\frac{11}{36}\right)^3
   \Gamma \left(\frac{5}{12}\right)^3 \Gamma \left(\frac{19}{36}\right)^3 \Gamma \left(\frac{23}{36}\right)^3 \Gamma
   \left(\frac{31}{36}\right)^3 \Gamma \left(\frac{35}{36}\right)^3}{\Gamma \left(\frac{1}{36}\right)^3 \Gamma
   \left(\frac{5}{36}\right)^3 \Gamma \left(\frac{13}{36}\right)^3 \Gamma \left(\frac{17}{36}\right)^3 \Gamma
   \left(\frac{7}{12}\right)^3 \Gamma \left(\frac{25}{36}\right)^3 \Gamma \left(\frac{29}{36}\right)^3 \Gamma
   \left(\frac{11}{12}\right)^3}$ & $\frac{1}{27}$ \\
  $\frac{\Gamma \left(\frac{1}{12}\right)^4 \Gamma \left(\frac{7}{36}\right)^4 \Gamma \left(\frac{11}{36}\right)^4
   \Gamma \left(\frac{5}{12}\right)^4 \Gamma \left(\frac{19}{36}\right)^4 \Gamma \left(\frac{23}{36}\right)^4 \Gamma
   \left(\frac{31}{36}\right)^4 \Gamma \left(\frac{35}{36}\right)^4}{\Gamma \left(\frac{1}{36}\right)^4 \Gamma
   \left(\frac{5}{36}\right)^4 \Gamma \left(\frac{13}{36}\right)^4 \Gamma \left(\frac{17}{36}\right)^4 \Gamma
   \left(\frac{7}{12}\right)^4 \Gamma \left(\frac{25}{36}\right)^4 \Gamma \left(\frac{29}{36}\right)^4 \Gamma
   \left(\frac{11}{12}\right)^4}$ & $\frac{1}{81}$ \\
  $-\frac{\Gamma \left(\frac{3}{28}\right) \Gamma \left(\frac{11}{28}\right) \Gamma \left(\frac{15}{28}\right)
   \Gamma \left(\frac{19}{28}\right) \Gamma \left(\frac{23}{28}\right) \Gamma \left(\frac{27}{28}\right)}{\Gamma
   \left(\frac{1}{28}\right) \Gamma \left(\frac{5}{28}\right) \Gamma \left(\frac{9}{28}\right) \Gamma
   \left(\frac{13}{28}\right) \Gamma \left(\frac{17}{28}\right) \Gamma \left(\frac{25}{28}\right)}$ & $\frac{\Gamma
   \left(-\frac{1}{4}\right)^2}{9 \sqrt{7} \Gamma \left(-\frac{3}{4}\right)^2}$ \\
  $\frac{\Gamma \left(\frac{3}{20}\right)^3 \Gamma \left(\frac{7}{20}\right)^3 \Gamma \left(\frac{11}{20}\right)^3
   \Gamma \left(\frac{19}{20}\right)^3}{\Gamma \left(\frac{1}{20}\right)^3 \Gamma \left(\frac{9}{20}\right)^3 \Gamma
   \left(\frac{13}{20}\right)^3 \Gamma \left(\frac{17}{20}\right)^3}$ & $\frac{1}{5 \sqrt{5}}$ \\
   $\frac{\Gamma \left(\frac{3}{20}\right)^4 \Gamma \left(\frac{7}{20}\right)^4 \Gamma \left(\frac{11}{20}\right)^4
   \Gamma \left(\frac{19}{20}\right)^4}{\Gamma \left(\frac{1}{20}\right)^4 \Gamma \left(\frac{9}{20}\right)^4 \Gamma
   \left(\frac{13}{20}\right)^4 \Gamma \left(\frac{17}{20}\right)^4}$ & $\frac{1}{25}$ \\
    $\frac{\Gamma \left(\frac{3}{20}\right)^5 \Gamma \left(\frac{7}{20}\right)^5 \Gamma \left(\frac{11}{20}\right)^5
   \Gamma \left(\frac{19}{20}\right)^5}{\Gamma \left(\frac{1}{20}\right)^5 \Gamma \left(\frac{9}{20}\right)^5 \Gamma
   \left(\frac{13}{20}\right)^5 \Gamma \left(\frac{17}{20}\right)^5}$ & $\frac{1}{25 \sqrt{5}}$ \\
    $\frac{\Gamma \left(\frac{1}{12}\right)^6 \Gamma \left(\frac{7}{36}\right)^6 \Gamma \left(\frac{11}{36}\right)^6
   \Gamma \left(\frac{5}{12}\right)^6 \Gamma \left(\frac{19}{36}\right)^6 \Gamma \left(\frac{23}{36}\right)^6 \Gamma
   \left(\frac{31}{36}\right)^6 \Gamma \left(\frac{35}{36}\right)^6}{\Gamma \left(\frac{1}{36}\right)^6 \Gamma
   \left(\frac{5}{36}\right)^6 \Gamma \left(\frac{13}{36}\right)^6 \Gamma \left(\frac{17}{36}\right)^6 \Gamma
   \left(\frac{7}{12}\right)^6 \Gamma \left(\frac{25}{36}\right)^6 \Gamma \left(\frac{29}{36}\right)^6 \Gamma
   \left(\frac{11}{12}\right)^6}$ & $\frac{1}{729}$ 
\end{tabular}
\caption{Table of quotient gamma functions}
\label{table:kysymys}
\end{table}
\end{center}
\begin{example} 
A Gosper relation for a $q$-trigonometric form in terms of quotient gamma functions given on page (80) in \cite{gosper} and entry (1) in Table \ref{table:kysymys}.
\begin{multline}
\prod_{n=0}^{\infty}\frac{(10 n+1) (10 n+3) (10 n+7) (10 n+9)}{(10 n+2) (10 n+4) (10 n+6) (10 n+8)}=\frac{\sin \left(\frac{3 \pi}{10}\right)}{2 \sin ^2\left(\frac{2 \pi }{5}\right)}\\
=\frac{\Gamma \left(\frac{3}{20}\right)
   \Gamma \left(\frac{7}{20}\right) \Gamma \left(\frac{11}{20}\right) \Gamma \left(\frac{19}{20}\right)}{\Gamma
   \left(\frac{1}{20}\right) \Gamma \left(\frac{9}{20}\right) \Gamma \left(\frac{13}{20}\right) \Gamma
   \left(\frac{17}{20}\right)}
   =\frac{1}{\sqrt{5}}
\end{multline}
\end{example}
\begin{example} 
Double finite product involving exponential of the Hurwitz-Lerch zeta function and quotient tangent functions. 
\begin{multline}\label{eq:bdh}
\prod_{p=0}^{n}\prod_{j=1}^{z}\exp \left(i (-1)^{j+z} (1+2 z)^{-p} \left(e^{i m (1+2 z)^{-1+p} (1-2 j+2 z)} \Phi \left(-e^{2 i m (1+2z)^p},1,\frac{1}{2}-\frac{j}{1+2 z}\right)\right.\right. \\ \left.\left.
+e^{i m (1+2 z)^{-1+p} (1+2 j+2 z)} \Phi \left(-e^{2 i m (1+2z)^p},1,\frac{1}{2}+\frac{j}{1+2 z}\right)\right.\right. \\ \left.\left.
-e^{i r (1+2 z)^{-1+p} (1-2 j+2 z)} \Phi \left(-e^{2 i r (1+2z)^p},1,\frac{1}{2}-\frac{j}{1+2 z}\right)\right.\right. \\ \left.\left.
-e^{i r (1+2 z)^{-1+p} (1+2 j+2 z)} \Phi \left(-e^{2 i r (1+2
z)^p},1,\frac{1}{2}+\frac{j}{1+2 z}\right)\right)\right)\\
\prod_{p=0}^{n}\left(\frac{\tan \left(\frac{1}{4} \left(\pi +2 m (1+2
   z)^p\right)\right)}{\tan \left(\frac{1}{4} \left(\pi +2 r (1+2z)^p\right)\right)}\right)^{-\left(\left(-1+(-1)^z\right) (1+2 z)^{-p}\right)}\\
=\left(\frac{\tan \left(\frac{2 m+\pi
   +2 \pi  z}{4+8 z}\right)}{\tan \left(\frac{\pi +2 r+2 \pi  z}{4+8 z}\right)}\right)^{-1-2 z} \left(\frac{\tan
   \left(\frac{1}{4} \left(\pi +2 m (1+2 z)^n\right)\right)}{\tan \left(\frac{1}{4} \left(\pi +2 r (1+2
   z)^n\right)\right)}\right)^{(1+2 z)^{-n}}
\end{multline}
\end{example}
\begin{proof}
Use equation (\ref{eq:theorem}) and form a second equation by replacing $m\to r$ and take the difference of both these equations then set $k=-1,a=1$ and simplify using entry (3) of Section (64:12) in \cite{atlas}.
\end{proof}
\begin{example} 
Finite product involving exponential of the Hurwitz-Lerch zeta function and quotient tangent functions. 
\begin{multline}\label{eq:bdh1}
\prod_{j=1}^{z}\exp \left(i (-1)^{j+z} \left(e^{i m \left(1-\frac{2 j}{1+2 z}\right)} \Phi \left(-e^{2 i
   m},1,\frac{1}{2}-\frac{j}{1+2 z}\right)\right.\right. \\ \left.\left.
+e^{i m \left(1+\frac{2 j}{1+2 z}\right)} \Phi \left(-e^{2 i
   m},1,\frac{1}{2}+\frac{j}{1+2 z}\right)-e^{i r \left(1-\frac{2 j}{1+2 z}\right)} \Phi \left(-e^{2 i
   r},1,\frac{1}{2}-\frac{j}{1+2 z}\right)\right.\right. \\ \left.\left.
-e^{i r \left(1+\frac{2 j}{1+2 z}\right)} \Phi \left(-e^{2 i
   r},1,\frac{1}{2}+\frac{j}{1+2 z}\right)\right)\right)\\
=\frac{\tan \left(\frac{1}{4} (2 m+\pi )\right)
   \left(\frac{\tan \left(\frac{2 m+\pi +2 \pi  z}{4+8 z}\right)}{\tan \left(\frac{\pi +2 r+2 \pi  z}{4+8
   z}\right)}\right)^{-1-2 z}}{\tan \left(\frac{1}{4} (\pi +2 r)\right) \left(\frac{\tan \left(\frac{1}{4} (2 m+\pi
   )\right)}{\tan \left(\frac{1}{4} (\pi +2 r)\right)}\right)^{1-(-1)^z}}
\end{multline}
\end{example}
\begin{proof}
Use equation (\ref{eq:bdh}) and set $p=n=0$ and simplify, where $z$ is any positive integer.
\end{proof}
\section{An example involving the exponential of the hypergeometric function in terms of reciprocal angles}
\begin{example} 
Exponential of the hypergeometric function involving reciprocal angles.
\begin{multline}
\exp \left(i \left(6 e^{\left.\frac{i}{3}\right/r} \, _2F_1\left(\frac{1}{6},1;\frac{7}{6};-e^{2 i/r}\right)-6 e^{\frac{i r}{3}} \, _2F_1\left(\frac{1}{6},1;\frac{7}{6};-e^{2 i r}\right)\right.\right. \\ \left.\left.
+\frac{6}{5}e^{\left.\frac{5 i}{3}\right/r} \, _2F_1\left(\frac{5}{6},1;\frac{11}{6};-e^{2 i/r}\right)-\frac{6}{5} e^{\frac{5 i r}{3}} \, _2F_1\left(\frac{5}{6},1;\frac{11}{6};-e^{2 i r}\right)\right)\right)\\
   =\tan\left(\frac{1}{4} (2 r+\pi )\right) \tan ^3\left(\frac{1}{12} (2 r+3 \pi )\right) \cot \left(\frac{1}{4} \left(\frac{2}{r}+\pi \right)\right) \cot ^3\left(\frac{1}{12} \left(\frac{2}{r}+3 \pi\right)\right)
\end{multline}
\end{example}
\begin{proof}
Use equation (\ref{eq:bdh1}) and replace $m$ by $1/r$ and set $z=1$ and simplify.
\end{proof}
\begin{figure}[H]
\includegraphics[scale=0.7]{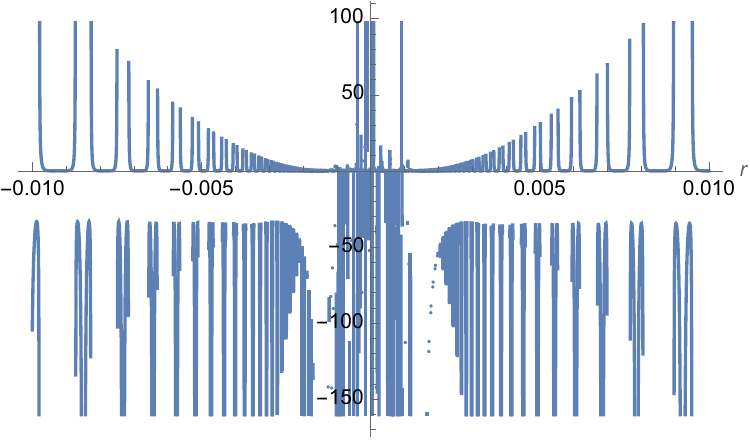}
\caption{Plot of  $f(r)=\frac{\tan ^3\left(\frac{\pi }{4}+\frac{r}{6}\right) \tan \left(\frac{1}{4} (\pi +2 r)\right)}{\tan ^3\left(\frac{\pi }{4}+\frac{1}{6 r}\right) \tan \left(\frac{1}{4} \left(\pi +\frac{2}{r}\right)\right)}$, $r\in\mathbb{R}$.}
   \label{fig:fig2}
\end{figure}
\vspace{-6pt}
\begin{figure}[H]
\includegraphics[scale=0.7]{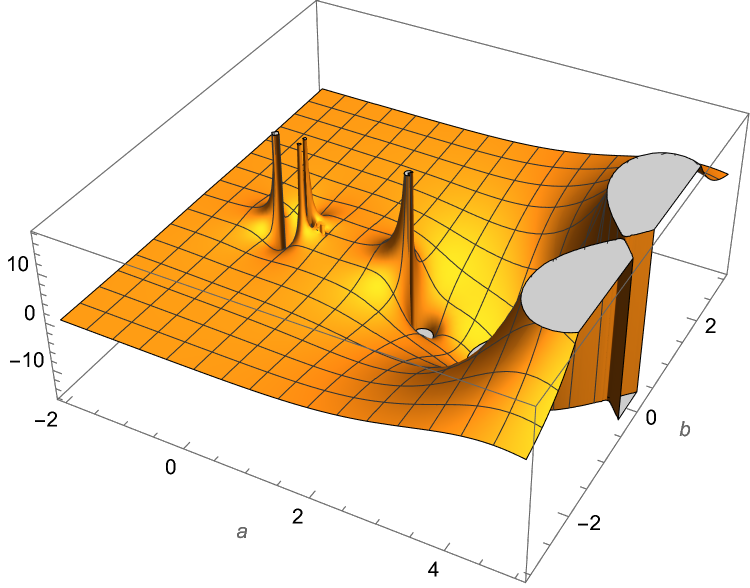}
\caption{Plot of  $f(r)=Re\left( \frac{\tan ^3\left(\frac{\pi }{4}+\frac{r}{6}\right) \tan \left(\frac{1}{4} (\pi +2 r)\right)}{\tan ^3\left(\frac{\pi }{4}+\frac{1}{6 r}\right) \tan \left(\frac{1}{4} \left(\pi +\frac{2}{r}\right)\right)}\right),r\in\mathbb{C}$.}
   \label{fig:fig3}
\end{figure}
\vspace{-6pt}
\begin{figure}[H]
\includegraphics[scale=0.7]{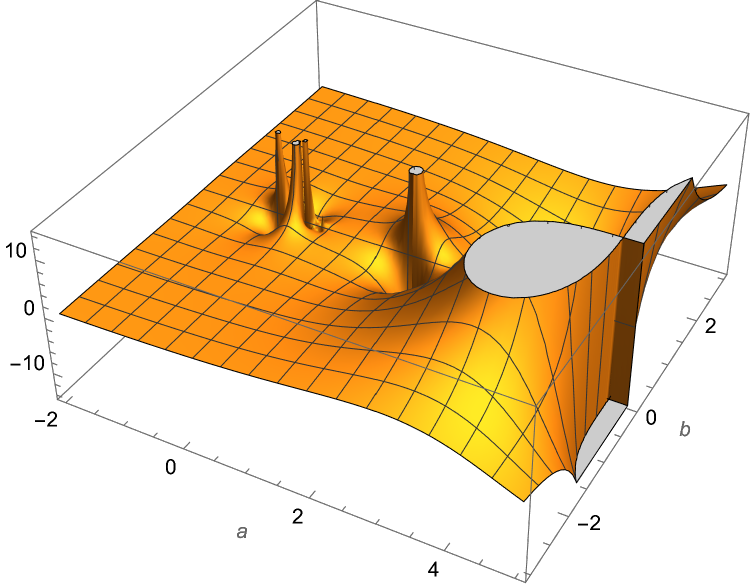}
\caption{Plot of  $f(r)=Im\left( \frac{\tan ^3\left(\frac{\pi }{4}+\frac{r}{6}\right) \tan \left(\frac{1}{4} (\pi +2 r)\right)}{\tan ^3\left(\frac{\pi }{4}+\frac{1}{6 r}\right) \tan \left(\frac{1}{4} \left(\pi +\frac{2}{r}\right)\right)}\right),r\in\mathbb{C}$.}
   \label{fig:fig4}
\end{figure}
\vspace{-6pt}
\begin{figure}[H]
\includegraphics[scale=0.7]{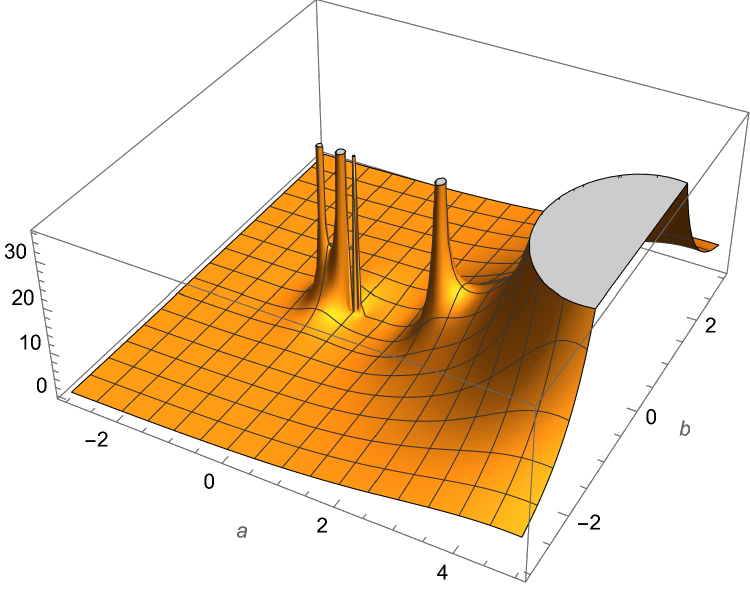}
\caption{Plot of  $f(r)=Abs\left( \frac{\tan ^3\left(\frac{\pi }{4}+\frac{r}{6}\right) \tan \left(\frac{1}{4} (\pi +2 r)\right)}{\tan ^3\left(\frac{\pi }{4}+\frac{1}{6 r}\right) \tan \left(\frac{1}{4} \left(\pi +\frac{2}{r}\right)\right)}\right),r\in\mathbb{C}$.}
   \label{fig:fig4}
\end{figure}
\vspace{-6pt}
\section{Determining the first partial derivative of the Hurwitz-Lerch zeta function}
In this example we derive a finite product expression which can be used to determine the first partial derivatives of the Hurwitz-Lerch zeta function with respect to the first parameter and the second parameter is zero.
\begin{example} 
Finite product of quotient gamma functions raised to complex power. 
\begin{multline}\label{eq:bdh2}
\prod_{j=1}^{z}\left(\frac{\Gamma \left(\frac{1}{4} \left(a-\frac{2 j}{2 z+1}+3\right)\right)}{\Gamma \left(\frac{1}{4} \left(a-\frac{2 j}{2 z+1}+1\right)\right)}\right)^{(-1)^{j+z} e^{-\frac{2 i \pi  j}{2 z+1}}}
   \left(\frac{\Gamma \left(\frac{1}{4} \left(a+\frac{2 j}{2 z+1}+3\right)\right)}{\Gamma \left(\frac{1}{4} \left(a+\frac{2 j}{2 z+1}+1\right)\right)}\right)^{(-1)^{j+z} e^{\frac{2 i \pi  j}{2 z+1}}}\\
=(4
   z+2)^{\frac{1}{2} \sec \left(\frac{\pi }{2 z+1}\right)} \left(\frac{\Gamma \left(\frac{a+1}{4}\right)}{\Gamma \left(\frac{a+3}{4}\right)}\right)^{(-1)^z} \exp \left(e^{\frac{i \pi }{2 z+1}}
   \Phi'\left(-e^{\frac{2 i \pi }{2 z+1}},0,\frac{1}{2} (2 a z+a+1)\right)\right)
\end{multline}
\end{example}
\begin{proof}
Use equation (\ref{eq:theorem}) and set $p=n=0,m=\pi$ and simplify in terms of the Hurwitz zeta function. Next take the first partial derivative with respect to $k$ and set $k=0$ and simplify in terms of the log-gamma function. Finally we take the exponential of both sides and simplify. Here $z$ is any positive integer greater than or equal to 1.
\end{proof}
\begin{example} 
An example in terms of the log-gamma function.
\begin{multline}
-(-1)^{2/3} \log \left(\frac{\Gamma \left(\frac{3}{8}\right) \left(\frac{\Gamma \left(\frac{5}{24}\right)}{\Gamma \left(\frac{17}{24}\right)}\right)^{\sqrt[3]{-1}} \left(\frac{\Gamma
   \left(\frac{25}{24}\right)}{\Gamma \left(\frac{13}{24}\right)}\right)^{(-1)^{2/3}}}{6 \Gamma \left(\frac{7}{8}\right)}\right)=\Phi'\left(-(-1)^{2/3},0,\frac{5}{4}\right)
\end{multline}
\end{example}
\begin{proof}
Use equation (\ref{eq:bdh2}) and set $a=1/2,z=1$ and simplify.
\end{proof}
%
%
%
\section{Discussion}
In this study, we employed a contour integration method to establish mathematical expressions for double finite sums involving the Hurwitz-Lerch zeta function. While this approach is generally easy, its application to the double finite sum of the secant function posed difficulties during evaluation. Our challenges encompassed both the simplification of the secant function's representation within the contour integral and the identification of specific values where the double product finite sum becomes simpler. The significance of this research lies in the deriving closed-form solutions through these methodologies. Consequently, we have introduced formulas to the existing body of knowledge, with the hope that they will prove valuable to the academic community

\section{Conclusion}
In this paper, we have presented a method for deriving double finite sum and products formulae involving the Hurwitz-Lerch zeta function along with some interesting special cases using both contour integration and well known algebraic techniques. We plan to apply these methods to derive other multiple sum and product formulae involving other special functions in future work.

\begin{thebibliography}{999}
%
\bibitem{prud}A. P. Prudnikov, Yu. A. Brychkov, and O. I. Marichev 
\emph{Integrals and Series: Elementary Functions}, Vol. 1. Gordon \& Breach Science Publishers, New York. \textbf{1986a }
%

\bibitem{prud3}A. P. Prudnikov, Yu. A. Brychkov, and O. I. Marichev 
\emph{Integrals and Series: More Special Functions}, Vol. 3. Gordon and Breach Science Publishers, New York. \textbf{1990}.
%

 \bibitem{reyn_ejpam}Reynolds, R., \& Stauffer, A.
\emph{A Note on the Infinite Sum of the Lerch function}. European Journal of Pure and Applied Mathematics, 15(1), 158–168. \textbf{2022}.
https://doi.org/10.29020/nybg.ejpam.v15i1.4137
%

\bibitem{reyn4} Reynolds, R.; Stauffer, A.
{A Method for Evaluating Definite Integrals in Terms of Special Functions with Examples}.  \emph{Int. Math. Forum} \textbf{2020}, \emph{15}, 235--244, 
doi:10.12988/imf.2020.91272 
%

\bibitem{hobson}Hobson, Ernest William. 
\emph{A Treatise on Plane Trigonometry}, United Kingdom: University Press, \textbf{1897}.
%

\bibitem{erd} Erd\'eyli, A.; Magnus, W.; Oberhettinger, F.; Tricomi, F.G.
\emph{Higher Transcendental Functions}; McGraw-Hill Book Company, Inc.: New York, NY, USA; Toronto, ON, Canada; London, UK, \textbf{1953}; Volume I.
%

\bibitem{grad} Gradshteyn, I.S.; Ryzhik, I.M.
\emph{Tables of Integrals, Series and Products}, 6th ed.; Academic Press: Cambridge, MA, USA, 
 \textbf{2000}.
%

\bibitem{gelca} Gelca, Răzvan., Andreescu, Titu.
\emph{Putnam and Beyond}. Germany: Springer International Publishing, \textbf{2017}.
 %
 
 \bibitem{atlas} Oldham, K.B.; Myland, J.C.; Spanier, J.
\emph{An Atlas of Functions: With Equator, the Atlas Function Calculator}, 2nd ed.; Springer: New York, NY, USA, \textbf{2009}.
%

\bibitem{krysicki}Krysicki, W. 
\emph{‘On Mieshalkin-Rogozin theorem and some properties of the second kind beta distribution’}, Discussiones Mathematicae Probability and Statistics, 20(2), p. 211. \textbf{2000}.
doi:10.7151/dmps.1012. 
%

\bibitem{kal}Kaluszka, M., Krysicki, W. 
\emph{On decompositions of some random variables}. Metrika 46, 159–175, \textbf{1997}. p. 214,
https://doi.org/10.1007/BF02717172
%

\bibitem{dlmf} Olver, F.W.J.; Lozier, D.W.; Boisvert, R.F.; Clark, C.W. (Eds.)
 \emph{NIST Digital Library of Mathematical Functions}; U.S. Department of Commerce, National Institute of Standards and Technology: Washington, DC, USA; Cambridge University Press: Cambridge, UK, \textbf{2010}; With 1 CD-ROM (Windows, Macintosh and UNIX). MR 2723248 (2012a:33001).
%

\bibitem{eric}Weisstein, Eric W.  
\emph{CRC Concise Encyclopedia of Mathematics}, CRC Press, \textbf{2002}, 3252 pages.
 %
 

\bibitem{guillera}Guillera, J. and Sondow, J. 
\emph{"Double Integrals and Infinite Products for Some Classical Constants Via Analytic Continuations of Lerch's Transcendent."}, 16 June \textbf{2005} 
http://arxiv.org/abs/math.NT/0506319.
%
\bibitem{gosper}Gosper, R.W.  
\emph{Experiments and Discoveries in q-Trigonometry}. In: Garvan, F.G., Ismail, M.E.H. (eds) Symbolic Computation, Number Theory, Special Functions, Physics and Combinatorics. Developments in Mathematics, vol 4. Springer, Boston, MA. \textbf{2001}.
https://doi.org/10.1007/978-1-4613-0257-5_6
%
\end{thebibliography}
\end{document}